\documentclass{amsart}
\usepackage{amsfonts}

\newtheorem{theorem}{Theorem}
\theoremstyle{plain}

\newtheorem{corollary}{Corollary}
\newtheorem{criterion}{Criterion}

\newtheorem{lemma}{Lemma}

\newtheorem{proposition}{Proposition}
\newtheorem{remark}{Remark}

\numberwithin{equation}{section}
\input{tcilatex}

\begin{document}
\title[inequalities for s-convex functions]{\textbf{Some Hadamard and
Simpson type integral inequalities via }$\mathbf{s}$\textbf{-convexity and
their applications}}
\author{Mevl\"{u}t TUN\c{C}}
\address{University of Kilis 7 Aral\i k, Department of Mathematics, Turkey}
\email{mevluttunc@kilis.edu.tr}
\date{}

\begin{abstract}
In this study, we establish and generalize some inequalities of Hadamard and
Simpson type based on $s$-convexity in the second sense. Some applications
to special means of positive real numbers are also given and generalized.
Examples are given to show the results. The results generalize the integral
inequalities in \cite{sari2} and \cite{xi}.
\end{abstract}

\maketitle

\section{$\mathbf{Introductions}$}

To establish analytic inequalities, one of the most efficient way is the
property of convexity of a dedicated function. Notedly, in the theory of
higher transcendental functions, there are many significant applications. We
can use the integral inequalities in order to study qualitative and
quantitative properties of integrals (see \cite{mit,nic,wu}). Thing
continuing to bewilder us by indicating new inferences, new difficulties and
also new open questions are a major mathematical outcome.

\bigskip \textbf{The Hermite-Hadamard inequality:} Let $f:I\subseteq 
\mathbb{R}
\rightarrow 
\mathbb{R}
$ be a convex function and $u,v\in I$ with $u<v$. The following double
inequality:%
\begin{equation}
f\left( \frac{u+v}{2}\right) \leq \frac{1}{v-u}\int_{u}^{v}f\left( x\right)
dx\leq \frac{f\left( u\right) +f\left( v\right) }{2}  \tag{HH}
\end{equation}%
is known in the literature as Hadamard's inequality (or Hermite-Hadamard
inequality) for convex functions. Keep in mind that some of the classical
inequalities for means can come from (HH) for convenient particular
selections of the function $f.$ If $f$ is concave, this double inequality
hold in the inversed way. See \cite{bes,nic} for details.

\textbf{The Simpson inequality:} The following inequality is well known in
the literature as Simpson's inequality; 
\begin{equation}
\left\vert \frac{1}{3}\left[ \frac{f\left( u\right) +f\left( v\right) }{2}%
+2f\left( \frac{u+v}{2}\right) \right] -\frac{1}{v-u}\int_{u}^{v}f\left(
x\right) dx\right\vert \leq \frac{1}{1280}\left\Vert f^{\left( 4\right)
}\right\Vert _{\infty }\left( v-u\right) ^{4},  \tag{S}
\end{equation}%
where the mapping $f:\left[ u,v\right] \rightarrow 
\mathbb{R}
$ is assumed to be four times continuously differentiable on the interval
and $f^{\left( 4\right) }$ to be bounded on $\left( u,v\right) $, that is, $%
\left\Vert f^{\left( 4\right) }\right\Vert _{\infty }=\sup_{t\in \left(
u,v\right) }\left\vert f^{\left( 4\right) }\left( t\right) \right\vert
<\infty $. See \cite{alo,ss,sari,sari2} for details.

In \cite{hud}, Hudzik and Maligranda considered among others the class of
functions which are $s$-convex in the second sense. This class is defined in
the following way: a function $f:%
\mathbb{R}
_{+}\rightarrow 
\mathbb{R}
$, where $%
\mathbb{R}
_{+}=\left[ 0,\infty \right) $, is said to be $s$-convex in the second sense
if%
\begin{equation*}
f\left( \alpha \lambda +\left( 1-\alpha \right) \mu \right) \leq \alpha
^{s}f\left( \lambda \right) +\left( 1-\alpha \right) ^{s}f\left( \mu \right)
\end{equation*}%
for all $\lambda ,\mu \in \left[ 0,\infty \right) ,$ $\alpha \in \left[ 0,1%
\right] $ and for some fixed $s\in \left( 0,1\right] $. This class of $s$%
-convex functions is usually denoted by $K_{s}^{2}$. It can be smoothly seen
that for $s=1$, $s$-convexity reduces to the ordinary convexity of functions
defined on $\left[ 0,\infty \right) $.

Recently, in \cite{sari, sari2}, Sarikaya \textit{et al.} presented the
important integral identity including the first-order derivative of $f$ to
establish many interesting Simpson-type inequalities for convex and $s$%
-convex functions.

Meanwhile, in \cite{xi}, Xi \textit{et al.} presented the following two
important integral identities including the first-order derivatives to
establish many interesting Hermite-Hadamard-type inequalities for convex
functions.

In this study, some new Hadamard and Simpson type integral inequalities for
differentiable functions are established, and are applied to produce some
inequalities of special means. Examples are given to show the results. The
results generalize the integral inequalities in \cite{sari2} and \cite{xi}.

\section{\textbf{Main results for }$s$\textbf{-convex functions}}

In order to demonstrate our main results, we need the following lemmas that
have been derived in \cite{xi}:

\begin{lemma}
\cite{xi}\label{l1}Let $f:I\subseteq 
\mathbb{R}
\rightarrow 
\mathbb{R}
$ be differentiable function on $I^{\circ }$, $u,v\in I,$ with $u<v.$ If $%
f^{\prime }\in L\left[ u,v\right] $ and $\lambda ,\mu \in 
\mathbb{R}
$ then%
\begin{eqnarray}
&&  \label{x} \\
&&\frac{\lambda f\left( u\right) +\mu f\left( v\right) }{2}+\frac{2-\lambda
-\mu }{2}f\left( \frac{u+v}{2}\right) -\frac{1}{v-u}\int_{u}^{v}f\left(
x\right) dx  \notag \\
&=&\frac{v-u}{4}\int_{0}^{1}\left[ \left( 1-\lambda -\alpha \right)
f^{\prime }\left( \alpha u+\left( 1-t\right) \frac{u+v}{2}\right) +\left(
\mu -\alpha \right) f^{\prime }\left( \alpha \frac{u+v}{2}+\left( 1-\alpha
\right) v\right) \right] d\alpha  \notag
\end{eqnarray}
\end{lemma}

\begin{lemma}
\cite{xi}\label{l2}For $x>0$ and $0\leq y\leq 1$, one has%
\begin{eqnarray}
\int_{0}^{1}\left\vert y-\alpha \right\vert ^{x}d\alpha &=&\frac{%
y^{x+1}+\left( 1-y\right) ^{x+1}}{x+1},  \label{y} \\
\int_{0}^{1}\alpha \left\vert y-\alpha \right\vert ^{x}d\alpha &=&\frac{%
y^{x+2}+\left( x+1+y\right) \left( 1-y\right) ^{x+1}}{\left( x+1\right)
\left( x+2\right) }.  \notag
\end{eqnarray}
\end{lemma}

\begin{theorem}
\label{t0}Let $f:I\subseteq 
\mathbb{R}
\rightarrow 
\mathbb{R}
$ be differentiable function on $I^{\circ }$, $u,v\in I,$ $0\leq \lambda ,$ $%
\mu \leq 1,$ and $f^{\prime }\in L\left[ u,v\right] .$ If $\left\vert
f^{\prime }\left( x\right) \right\vert ^{r}$ is $s$-convex function in the
second sense on $\left[ u,v\right] $ for some fixed $s\in \left( 0,1\right]
, $ $p,r>1,$ $1/p+1/r=1$, then%
\begin{eqnarray}
&&  \label{xx} \\
&&\left\vert \frac{\lambda f\left( u\right) +\mu f\left( v\right) }{2}+\frac{%
2-\lambda -\mu }{2}f\left( \frac{u+v}{2}\right) -\frac{1}{v-u}%
\int_{u}^{v}f\left( x\right) dx\right\vert  \notag \\
&\leq &\frac{v-u}{4}\left( \frac{\left( 1-\lambda \right) ^{p+1}+\lambda
^{p+1}}{p+1}\right) ^{\frac{1}{p}}\left( \frac{2^{s+1}-1}{2^{s}\left(
s+1\right) }\left\vert f^{\prime }\left( u\right) \right\vert ^{r}+\frac{1}{%
2^{s}\left( s+1\right) }\left\vert f^{\prime }\left( v\right) \right\vert
^{r}\right) ^{\frac{1}{r}}  \notag \\
&&+\frac{v-u}{4}\left( \frac{\mu ^{p+1}+\left( 1-\mu \right) ^{p+1}}{p+1}%
\right) ^{\frac{1}{p}}\left( \frac{1}{2^{s}\left( s+1\right) }\left\vert
f^{\prime }\left( u\right) \right\vert ^{r}+\frac{2^{s+1}-1}{2^{s}\left(
s+1\right) }\left\vert f^{\prime }\left( v\right) \right\vert ^{r}\right) ^{%
\frac{1}{r}}.  \notag
\end{eqnarray}
\end{theorem}

\begin{proof}
Assume that $p>1,$ by Lemma \ref{l1} and using the well known H\"{o}lder
inequality, we have 
\begin{eqnarray*}
&&\left\vert \frac{\lambda f\left( u\right) +\mu f\left( v\right) }{2}+\frac{%
2-\lambda -\mu }{2}f\left( \frac{u+v}{2}\right) -\frac{1}{v-u}%
\int_{u}^{v}f\left( x\right) dx\right\vert \\
&\leq &\frac{v-u}{4}\left[ \int_{0}^{1}\left\vert 1-\lambda -\alpha
\right\vert \left\vert f^{\prime }\left( \alpha u+\left( 1-\alpha \right) 
\frac{u+v}{2}\right) \right\vert d\alpha +\int_{0}^{1}\left\vert \mu -\alpha
\right\vert \left\vert f^{\prime }\left( \alpha \frac{u+v}{2}+\left(
1-\alpha \right) v\right) \right\vert d\alpha \right] \\
&\leq &\frac{v-u}{4}\left( \int_{0}^{1}\left\vert 1-\lambda -\alpha
\right\vert ^{p}dt\right) ^{\frac{1}{p}}\left( \int_{0}^{1}\left\vert
f^{\prime }\left( \alpha u+\left( 1-\alpha \right) \frac{u+v}{2}\right)
\right\vert ^{r}d\alpha \right) ^{\frac{1}{r}} \\
&&+\frac{v-u}{4}\left( \int_{0}^{1}\left\vert \mu -\alpha \right\vert
^{p}d\alpha \right) ^{\frac{1}{p}}\left( \int_{0}^{1}\left\vert f^{\prime
}\left( \alpha \frac{u+v}{2}+\left( 1-\alpha \right) v\right) \right\vert
^{r}d\alpha \right) ^{\frac{1}{r}}.
\end{eqnarray*}%
Since $\left\vert f^{\prime }\left( x\right) \right\vert ^{r}$ is $s$-convex
in the second sense on $\left[ u,v\right] ,$ then we get%
\begin{eqnarray*}
\int_{0}^{1}\left\vert f^{\prime }\left( \alpha u+\left( 1-\alpha \right) 
\frac{u+v}{2}\right) \right\vert ^{r}d\alpha &\leq &\int_{0}^{1}\left(
\left( \frac{1+\alpha }{2}\right) ^{s}\left\vert f^{\prime }\left( u\right)
\right\vert ^{r}+\left( \frac{1-\alpha }{2}\right) ^{s}\left\vert f^{\prime
}\left( v\right) \right\vert ^{r}\right) d\alpha \\
&=&\frac{2^{s+1}-1}{2^{s}\left( s+1\right) }\left\vert f^{\prime }\left(
u\right) \right\vert ^{r}+\frac{1}{2^{s}\left( s+1\right) }\left\vert
f^{\prime }\left( v\right) \right\vert ^{r}
\end{eqnarray*}%
and%
\begin{eqnarray*}
\int_{0}^{1}\left\vert f^{\prime }\left( \alpha \frac{u+v}{2}+\left(
1-\alpha \right) v\right) \right\vert ^{r}d\alpha &\leq &\int_{0}^{1}\left(
\left( \frac{\alpha }{2}\right) ^{s}\left\vert f^{\prime }\left( u\right)
\right\vert ^{r}+\left( \frac{2-\alpha }{2}\right) ^{s}\left\vert f^{\prime
}\left( v\right) \right\vert ^{r}\right) d\alpha \\
&=&\frac{1}{2^{s}\left( s+1\right) }\left\vert f^{\prime }\left( u\right)
\right\vert ^{r}+\frac{2^{s+1}-1}{2^{s}\left( s+1\right) }\left\vert
f^{\prime }\left( v\right) \right\vert ^{r}
\end{eqnarray*}%
where we have used the fact that%
\begin{equation*}
\int_{0}^{1}\left\vert 1-\lambda -\alpha \right\vert ^{p}d\alpha =\frac{%
\left( 1-\lambda \right) ^{p+1}+\lambda ^{p+1}}{p+1}\text{ and }%
\int_{0}^{1}\left\vert \mu -\alpha \right\vert ^{p}d\alpha =\frac{\mu
^{p+1}+\left( 1-\mu \right) ^{p+1}}{p+1}.
\end{equation*}%
This completes the proof.
\end{proof}

If taking $\lambda =%
\mu
$ in Theorem \ref{t0}, we derive the following corollary.

\begin{corollary}
Let $f:I\subseteq 
\mathbb{R}
\rightarrow 
\mathbb{R}
$ be differentiable function on $I^{\circ }$, $u,v\in I,$ $0\leq \lambda
\leq 1,$ and $f^{\prime }\in L\left[ u,v\right] .$ If $\left\vert f^{\prime
}\left( x\right) \right\vert ^{r}$ is $s$-convex function in the second
sense on $\left[ u,v\right] $ for some fixed $s\in \left( 0,1\right] ,$ and $%
1/p+1/r=1$, then%
\begin{eqnarray}
&& \\
&&\left\vert \lambda \frac{f\left( u\right) +f\left( v\right) }{2}+\left(
1-\lambda \right) f\left( \frac{u+v}{2}\right) -\frac{1}{v-u}%
\int_{u}^{v}f\left( x\right) dx\right\vert  \notag \\
&\leq &\frac{v-u}{4}\left( \frac{\left( 1-\lambda \right) ^{p+1}+\lambda
^{p+1}}{p+1}\right) ^{\frac{1}{p}}\times \left\{ \left( \frac{2^{s+1}-1}{%
2^{s}\left( s+1\right) }\left\vert f^{\prime }\left( u\right) \right\vert
^{r}+\frac{1}{2^{s}\left( s+1\right) }\left\vert f^{\prime }\left( v\right)
\right\vert ^{r}\right) ^{\frac{1}{r}}\right.  \notag \\
&&\left. +\left( \frac{1}{2^{s}\left( s+1\right) }\left\vert f^{\prime
}\left( u\right) \right\vert ^{r}+\frac{2^{s+1}-1}{2^{s}\left( s+1\right) }%
\left\vert f^{\prime }\left( v\right) \right\vert ^{r}\right) ^{\frac{1}{r}%
}\right\}  \notag
\end{eqnarray}
\end{corollary}

\begin{remark}
In Theorem \ref{t0}, if we take $\lambda =%
\mu
=1/3,$ then Theorem \ref{t0} reduces to \cite[Teorem 9]{sari2}. Hence, the
result in Theorem \ref{t0} is generalizations of the results of Sarikaya et
al. in \cite{sari2}.
\end{remark}

If we take $s=1$ and $\lambda =%
\mu
=1/2,2/3,1/3$, respectively, in Theorem \ref{t0}, the following inequalities
can be deduced.

\begin{corollary}
Let $f:I\subseteq 
\mathbb{R}
\rightarrow 
\mathbb{R}
$ be a differentiable function on $I^{\circ }$, $u,v\in I$ with $u<v,$ and $%
f^{\prime }\in L\left[ u,v\right] .$ If $\left\vert f^{\prime }\left(
x\right) \right\vert ^{r}$ is convex function on $\left[ u,v\right] $ for $%
1/p+1/r=1$, then%
\begin{eqnarray}
&& \\
&&\left\vert \frac{1}{2}\left[ \frac{f\left( u\right) +f\left( v\right) }{2}%
+f\left( \frac{u+v}{2}\right) \right] -\frac{1}{v-u}\int_{u}^{v}f\left(
x\right) dx\right\vert  \notag \\
&\leq &\frac{v-u}{8\left( p+1\right) ^{1/p}4^{1/r}}\times \left[ \left(
3\left\vert f^{\prime }\left( u\right) \right\vert ^{r}+\left\vert f^{\prime
}\left( v\right) \right\vert ^{r}\right) ^{\frac{1}{r}}+\left( \left\vert
f^{\prime }\left( u\right) \right\vert ^{r}+3\left\vert f^{\prime }\left(
v\right) \right\vert ^{r}\right) ^{\frac{1}{r}}\right]  \notag
\end{eqnarray}%
\begin{eqnarray}
&& \\
&&\left\vert \frac{1}{3}\left[ f\left( u\right) +f\left( v\right) +f\left( 
\frac{u+v}{2}\right) \right] -\frac{1}{v-u}\int_{u}^{v}f\left( x\right)
dx\right\vert  \notag \\
&\leq &\frac{v-u}{4^{1+1/r}}\left( \frac{\left( 1+2^{p+1}\right) }{%
3^{p+1}\left( p+1\right) }\right) ^{\frac{1}{p}}\times \left[ \left(
3\left\vert f^{\prime }\left( u\right) \right\vert ^{r}+\left\vert f^{\prime
}\left( v\right) \right\vert ^{r}\right) ^{\frac{1}{r}}+\left( \left\vert
f^{\prime }\left( u\right) \right\vert ^{r}+3\left\vert f^{\prime }\left(
v\right) \right\vert ^{r}\right) ^{\frac{1}{r}}\right]  \notag
\end{eqnarray}%
\begin{eqnarray}
&&  \label{ss} \\
&&\left\vert \frac{1}{6}\left[ f\left( u\right) +f\left( v\right) +4f\left( 
\frac{u+v}{2}\right) \right] -\frac{1}{v-u}\int_{u}^{v}f\left( x\right)
dx\right\vert  \notag \\
&\leq &\frac{v-u}{4^{1+1/r}}\left( \frac{\left( 1+2^{p+1}\right) }{%
3^{p+1}\left( p+1\right) }\right) ^{\frac{1}{p}}\times \left[ \left(
3\left\vert f^{\prime }\left( u\right) \right\vert ^{r}+\left\vert f^{\prime
}\left( v\right) \right\vert ^{r}\right) ^{\frac{1}{r}}+\left( \left\vert
f^{\prime }\left( u\right) \right\vert ^{r}+3\left\vert f^{\prime }\left(
v\right) \right\vert ^{r}\right) ^{\frac{1}{r}}\right]  \notag
\end{eqnarray}
\end{corollary}

If taking $\lambda =%
\mu
=1/2$, in Theorem \ref{t0}, the following inequalities can be deduced.

\begin{corollary}
Let $f:I\subseteq 
\mathbb{R}
\rightarrow 
\mathbb{R}
$ be a differentiable function on $I^{\circ }$, $u,v\in I,$ and $f^{\prime
}\in L\left[ u,v\right] .$ If $\left\vert f^{\prime }\left( x\right)
\right\vert ^{r}$ is $s$-convex function in the second sense on $\left[ u,v%
\right] $ for some fixed $s\in \left( 0,1\right] ,$ and $1/p+1/r=1$ and%
\begin{equation*}
\frac{f\left( u\right) +f\left( v\right) }{2}=f\left( \frac{u+v}{2}\right)
\end{equation*}%
then%
\begin{eqnarray}
&&\left\vert \frac{f\left( u\right) +f\left( v\right) }{2}-\frac{1}{v-u}%
\int_{u}^{v}f\left( x\right) dx\right\vert \\
&=&\left\vert f\left( \frac{u+v}{2}\right) -\frac{1}{v-u}\int_{u}^{v}f\left(
x\right) dx\right\vert  \notag \\
&\leq &\frac{v-u}{8\left( p+1\right) ^{\frac{1}{p}}}\left[ \left( \frac{%
2^{s+1}-1}{2^{s}\left( s+1\right) }\left\vert f^{\prime }\left( u\right)
\right\vert ^{r}+\frac{1}{2^{s}\left( s+1\right) }\left\vert f^{\prime
}\left( v\right) \right\vert ^{r}\right) ^{\frac{1}{r}}\right.  \notag \\
&&\left. +\left( \frac{1}{2^{s}\left( s+1\right) }\left\vert f^{\prime
}\left( u\right) \right\vert ^{r}+\frac{2^{s+1}-1}{2^{s}\left( s+1\right) }%
\left\vert f^{\prime }\left( v\right) \right\vert ^{r}\right) ^{\frac{1}{r}}%
\right]  \notag
\end{eqnarray}
\end{corollary}

\begin{corollary}
In Theorem \ref{t0}, when $s,$ $\lambda ,$ $\mu ,$ $p,$ $r$ are taken like
in the following, we obtain

i) For $s=0.3,\lambda =0.3,\mu =0.3,p=2,r=2$ in (\ref{xx}), 
\begin{eqnarray*}
&&\left\vert \frac{3}{10}\frac{f\left( u\right) +f\left( v\right) }{2}+\frac{%
14}{20}f\left( \frac{u+v}{2}\right) -\frac{1}{b-u}\int_{u}^{v}f\left(
x\right) dx\right\vert \\
&\leq &\frac{0.351\left( v-u\right) }{4}\left[ \left( 0.914\left\vert
f^{\prime }\left( u\right) \right\vert ^{2}+0.625\left\vert f^{\prime
}\left( v\right) \right\vert ^{2}\right) ^{\frac{1}{2}}+\left(
0.625\left\vert f^{\prime }\left( u\right) \right\vert ^{r}+0.914\left\vert
f^{\prime }\left( v\right) \right\vert ^{2}\right) ^{\frac{1}{2}}\right]
\end{eqnarray*}%
$\ $

ii) Taking $s=0.5,\lambda =0.5,\mu =0.5,p=2,r=2$ in (\ref{xx}) gives;%
\begin{eqnarray*}
&&\left\vert \frac{1}{2}\left[ \frac{f\left( u\right) +f\left( v\right) }{2}%
+f\left( \frac{u+v}{2}\right) \right] -\frac{1}{v-u}\int_{u}^{v}f\left(
x\right) dx\right\vert \\
&\leq &\frac{0.289\left( v-u\right) }{4}\left[ \left( 0.862\left\vert
f^{\prime }\left( u\right) \right\vert ^{2}+0.471\left\vert f^{\prime
}\left( v\right) \right\vert ^{2}\right) ^{\frac{1}{2}}+\left(
0.471\left\vert f^{\prime }\left( u\right) \right\vert ^{2}+0.862\left\vert
f^{\prime }\left( v\right) \right\vert ^{2}\right) ^{\frac{1}{2}}\right]
\end{eqnarray*}

iii) Taking $s=0.75,\lambda =0.3,\mu =0.7,p=10,r=10/9$ in (\ref{xx}) gives;%
\begin{eqnarray*}
&&\left\vert \frac{3f\left( u\right) +7f\left( v\right) }{20}+\frac{1}{2}%
f\left( \frac{u+v}{2}\right) -\frac{1}{v-u}\int_{u}^{v}f\left( x\right)
dx\right\vert \\
&\leq &\frac{0.531\left( v-u\right) }{4}\left[ \left( 0.803\left\vert
f^{\prime }\left( u\right) \right\vert ^{10/9}+\allowbreak 0.34\left\vert
f^{\prime }\left( v\right) \right\vert ^{10/9}\right) ^{\frac{9}{10}}+\left(
0.34\left\vert f^{\prime }\left( u\right) \right\vert
^{10/9}+0.803\left\vert f^{\prime }\left( v\right) \right\vert
^{10/9}\right) ^{\frac{9}{10}}\right]
\end{eqnarray*}

iv) Taking $s=0.4,\lambda =0.2,\mu =0.8,p=3,r=3/2$ in (\ref{xx}) gives;%
\begin{eqnarray*}
&&\left\vert \frac{f\left( u\right) +4f\left( v\right) }{10}+\frac{1}{2}%
f\left( \frac{u+v}{2}\right) -\frac{1}{v-u}\int_{u}^{v}f\left( x\right)
dx\right\vert \\
&\leq &\frac{0.468\left( v-u\right) }{4}\allowbreak \left[ \left(
0.887\left\vert f^{\prime }\left( u\right) \right\vert
^{3/2}+0.541\left\vert f^{\prime }\left( v\right) \right\vert ^{3/2}\right)
^{\frac{2}{3}}+\left( 0.541\left\vert f^{\prime }\left( u\right) \right\vert
^{3/2}+0.887\left\vert f^{\prime }\left( v\right) \right\vert ^{3/2}\right)
^{\frac{2}{3}}\right]
\end{eqnarray*}

v) Taking $s=0.4,\lambda =0.2,\mu =0.8,p=e,r=e/\left( e-1\right) $ in (\ref%
{xx}) gives;%
\begin{eqnarray*}
&&\left\vert \frac{f\left( u\right) +4f\left( v\right) }{10}+\frac{1}{2}%
f\left( \frac{u+v}{2}\right) -\frac{1}{v-u}\int_{u}^{v}f\left( x\right)
dx\right\vert \\
&\leq &\frac{0.455\left( v-u\right) }{4}\allowbreak \left[ \left(
0.887\left\vert f^{\prime }\left( u\right) \right\vert ^{e/\left( e-1\right)
}+0.541\left\vert f^{\prime }\left( v\right) \right\vert ^{e/\left(
e-1\right) }\right) ^{\frac{e}{e-1}}\right. \\
&&\left. +\left( 0.541\left\vert f^{\prime }\left( u\right) \right\vert
^{e/\left( e-1\right) }+0.887\left\vert f^{\prime }\left( v\right)
\right\vert ^{e/\left( e-1\right) }\right) ^{\frac{e}{e-1}}\right]
\end{eqnarray*}

vi) Taking $s=1,\lambda =1/3,\mu =2/3,p=r=2$ in (\ref{xx}) gives;%
\begin{eqnarray*}
&&\left\vert \frac{f\left( u\right) +2f\left( v\right) }{6}+\frac{1}{2}%
f\left( \frac{u+v}{2}\right) -\frac{1}{v-u}\int_{u}^{v}f\left( x\right)
dx\right\vert \\
&\leq &\frac{v-u}{12}\left[ \left( \frac{3}{4}\left\vert f^{\prime }\left(
u\right) \right\vert ^{2}+\frac{1}{4}\left\vert f^{\prime }\left( v\right)
\right\vert ^{2}\right) ^{\frac{1}{2}}+\left( \frac{1}{4}\left\vert
f^{\prime }\left( u\right) \right\vert ^{2}+\frac{3}{4}\left\vert f^{\prime
}\left( v\right) \right\vert ^{2}\right) ^{\frac{1}{2}}\right]
\end{eqnarray*}%
etc.
\end{corollary}

\begin{theorem}
\label{t1}Let $f:I\subseteq 
\mathbb{R}
\rightarrow 
\mathbb{R}
$ be differentiable function on $I^{\circ }$, $u,v\in I,$ $0\leq \lambda ,$ $%
\mu \leq 1,$ and $f^{\prime }\in L\left[ u,v\right] .$ If $\left\vert
f^{\prime }\left( x\right) \right\vert ^{r}$ is $s$-convex function in the
second sense on $\left[ u,v\right] $ for $r\geq 1$, then%
\begin{eqnarray}
&&\left\vert \frac{\lambda f\left( u\right) +\mu f\left( v\right) }{2}+\frac{%
2-\lambda -\mu }{2}f\left( \frac{u+v}{2}\right) -\frac{1}{v-u}%
\int_{u}^{v}f\left( x\right) dx\right\vert  \label{20} \\
&\leq &\frac{v-u}{8}\left( \frac{1}{2^{s-1}\left( s+1\right) \left(
s+2\right) }\right) ^{1/r}  \notag \\
&&\left\{ \left( \left( 2\lambda ^{2}-2\lambda +1\right) \right)
^{1-1/r}\times \left[ E\left\vert f^{\prime }\left( u\right) \right\vert
^{r}+L\left\vert f^{\prime }\left( v\right) \right\vert ^{r}\right]
^{1/r}\right.  \notag \\
&&+\left. \left( \left( 2\mu ^{2}-2\mu +1\right) \right) ^{1-1/r}\times 
\left[ I\left\vert f^{\prime }\left( u\right) \right\vert ^{r}+F\left\vert
f^{\prime }\left( v\right) \right\vert ^{r}\right] ^{1/r}\right\}  \notag
\end{eqnarray}%
where%
\begin{eqnarray*}
E &=&\left( 2\left( 2-\lambda \right) ^{s+2}+\left( \lambda -1\right) \left(
s+2^{s+2}+2\right) +2^{s+1}s\lambda -1\right) \allowbreak \\
L &=&\left( 2\lambda ^{s+2}+s\left( 1-\lambda \right) -2\lambda +1\right) \\
I &=&\left( 2\mu ^{s+2}+s\left( 1-\mu \right) -2\mu +1\right) \allowbreak \\
F &=&\left( 2\left( 2-\mu \right) ^{s+2}+\left( \mu -1\right) \left(
s+2^{s+2}+2\right) +2^{s+1}s\mu -1\right) \allowbreak .
\end{eqnarray*}
\end{theorem}

\begin{proof}
For $r\geq 1,$ from Lemma \ref{l1}, by using the $s$-convexity of $%
\left\vert f^{\prime }\left( x\right) \right\vert ^{r}$ on $\left[ u,v\right]
,$ and the famous power mean inequality, we can write 
\begin{eqnarray*}
&&\left\vert \frac{\lambda f\left( u\right) +\mu f\left( v\right) }{2}+\frac{%
2-\lambda -\mu }{2}f\left( \frac{u+v}{2}\right) -\frac{1}{v-u}%
\int_{u}^{v}f\left( x\right) dx\right\vert \\
&\leq &\frac{v-u}{4}\left[ \int_{0}^{1}\left\vert 1-\lambda -\alpha
\right\vert \left\vert f^{\prime }\left( \alpha u+\left( 1-\alpha \right) 
\frac{u+v}{2}\right) \right\vert d\alpha +\int_{0}^{1}\left\vert \mu -\alpha
\right\vert \left\vert f^{\prime }\left( \alpha \frac{u+v}{2}+\left(
1-\alpha \right) v\right) \right\vert d\alpha \right] \\
&\leq &\frac{v-u}{4}\left\{ \left( \int_{0}^{1}\left\vert 1-\lambda -\alpha
\right\vert d\alpha \right) ^{1-1/r}\left[ \int_{0}^{1}\left\vert 1-\lambda
-\alpha \right\vert \left( \left( \frac{1+\alpha }{2}\right) ^{s}\left\vert
f^{\prime }\left( u\right) \right\vert ^{r}+\left( \frac{1-\alpha }{2}%
\right) ^{s}\left\vert f^{\prime }\left( v\right) \right\vert ^{r}\right)
d\alpha \right] ^{1/r}\right. \\
&&+\left. \left( \int_{0}^{1}\left\vert \mu -\alpha \right\vert d\alpha
\right) ^{1-1/r}\left[ \int_{0}^{1}\left\vert \mu -\alpha \right\vert \left(
\left( \frac{\alpha }{2}\right) ^{s}\left\vert f^{\prime }\left( u\right)
\right\vert ^{r}+\left( \frac{2-\alpha }{2}\right) ^{s}\left\vert f^{\prime
}\left( v\right) \right\vert ^{r}\right) dt\right] ^{1/r}\right\}
\end{eqnarray*}%
By direct calculation, we obtain%
\begin{eqnarray*}
&&\int_{0}^{1}\left\vert 1-\lambda -\alpha \right\vert \left( \left( \frac{%
1+\alpha }{2}\right) ^{s}\left\vert f^{\prime }\left( u\right) \right\vert
^{r}+\left( \frac{1-\alpha }{2}\right) ^{s}\left\vert f^{\prime }\left(
v\right) \right\vert ^{r}\right) d\alpha \\
&=&\frac{\left\vert f^{\prime }\left( u\right) \right\vert ^{r}}{2^{s}}%
\int_{0}^{1}\left\vert 1-\lambda -\alpha \right\vert \left( 1+\alpha \right)
^{s}d\alpha +\frac{\left\vert f^{\prime }\left( v\right) \right\vert ^{r}}{%
2^{s}}\int_{0}^{1}\left\vert 1-\lambda -\alpha \right\vert \left( 1-\alpha
\right) ^{s}d\alpha \\
&=&\frac{\left\vert f^{\prime }\left( u\right) \right\vert ^{r}}{2^{s}\left(
s+1\right) \left( s+2\right) }\left( 2\left( 2-\lambda \right) ^{s+2}+\left(
\lambda -1\right) \left( s+2^{s+2}+2\right) +2^{s+1}s\lambda -1\right)
\allowbreak \\
&&+\frac{\left\vert f^{\prime }\left( v\right) \right\vert ^{r}}{2^{s}\left(
s+1\right) \left( s+2\right) }\left( 2\lambda ^{s+2}+s\left( 1-\lambda
\right) -2\lambda +1\right)
\end{eqnarray*}%
Similarly, we have%
\begin{eqnarray*}
&&\int_{0}^{1}\left\vert \mu -\alpha \right\vert \left( \left( \frac{\alpha 
}{2}\right) ^{s}\left\vert f^{\prime }\left( u\right) \right\vert
^{r}+\left( \frac{2-\alpha }{2}\right) ^{s}\left\vert f^{\prime }\left(
v\right) \right\vert ^{r}\right) d\alpha \\
&=&\frac{\left\vert f^{\prime }\left( u\right) \right\vert ^{r}}{2^{s}}%
\int_{0}^{1}\left\vert \mu -\alpha \right\vert \alpha ^{s}dt+\frac{%
\left\vert f^{\prime }\left( v\right) \right\vert ^{r}}{2^{s}}%
\int_{0}^{1}\left\vert \mu -\alpha \right\vert \left( 2-\alpha \right)
^{s}d\alpha \\
&=&\frac{\left\vert f^{\prime }\left( u\right) \right\vert ^{r}}{2^{s}\left(
s+1\right) \left( s+2\right) }\left( 2\mu ^{s+2}+s\left( 1-\mu \right) -2\mu
+1\right) \allowbreak \\
&&+\frac{\left\vert f^{\prime }\left( v\right) \right\vert ^{r}}{2^{s}\left(
s+1\right) \left( s+2\right) }\left( 2\left( 2-\mu \right) ^{s+2}+\left( \mu
-1\right) \left( s+2^{s+2}+2\right) +2^{s+1}s\mu -1\right) \allowbreak
\end{eqnarray*}%
\begin{equation*}
\int_{0}^{1}\left\vert 1-\lambda -\alpha \right\vert d\alpha =\frac{1}{2}%
\left( 2\lambda ^{2}-2\lambda +1\right)
\end{equation*}%
\begin{equation*}
\int_{0}^{1}\left\vert \mu -\alpha \right\vert d\alpha =\frac{1}{2}\left(
2\mu ^{2}-2\mu +1\right)
\end{equation*}%
Replace with the above four equalities into the inequality and the proof is
finished.
\end{proof}

\begin{remark}
In Theorem \ref{t1}, if we take $s=1,$ then Theorem \ref{t1} reduces to \cite%
[Teorem 3.1]{xi}. Hence, the result in Theorem \ref{t1} is generalizations
of the results of Xi et al. in \cite[Theorem 3.1]{xi}.
\end{remark}

If taking $\lambda =%
\mu
$ in (\ref{20}), we derive the following corollary.

\begin{corollary}
Let $f:I\subseteq 
\mathbb{R}
\rightarrow 
\mathbb{R}
$ be differentiable function on $I^{\circ }$, $u,v\in I,$ $0\leq \lambda
\leq 1,$ and $f^{\prime }\in L\left[ u,v\right] .$ If $\left\vert f^{\prime
}\left( x\right) \right\vert ^{r}$ is $s$-convex function in the second
sense on $\left[ u,v\right] $ for some fixed $s\in \left( 0,1\right] ,$ and $%
r\geq 1$, then%
\begin{eqnarray}
&&\left\vert \frac{\lambda f\left( u\right) +\lambda f\left( v\right) }{2}+%
\frac{2-2\lambda }{2}f\left( \frac{u+v}{2}\right) -\frac{1}{v-u}%
\int_{u}^{v}f\left( x\right) dx\right\vert \\
&\leq &\frac{v-u}{8\left( 2^{s-1}\left( s+1\right) \left( s+2\right) \right)
^{1/r}}\left( \left( 2\lambda ^{2}-2\lambda +1\right) \right) ^{1-1/r} 
\notag \\
&&\times \left\{ \left[ \left\vert f^{\prime }\left( u\right) \right\vert
^{r}E+\left\vert f^{\prime }\left( v\right) \right\vert ^{r}L\right] ^{1/r}+%
\left[ \left\vert f^{\prime }\left( u\right) \right\vert ^{r}L+\left\vert
f^{\prime }\left( v\right) \right\vert ^{r}E\right] ^{1/r}\right\}  \notag
\end{eqnarray}%
where $E$ and $L$ are like above.
\end{corollary}

If taking $\lambda =%
\mu
=1/2,2/3,1/3$, respectively, in Theorem \ref{t1}, the following inequalities
can be deduced.

\begin{corollary}
Let $s\in \left( 0,1\right] $ and $f:I\subseteq 
\mathbb{R}
\rightarrow 
\mathbb{R}
$ be differentiable function on $I^{\circ }$, $u,v\in I$ with $u<v,$ and $%
f^{\prime }\in L\left[ u,v\right] .$ If $\left\vert f^{\prime }\left(
x\right) \right\vert ^{r}$ is $s$-convex function in the second sense on $%
\left[ u,v\right] $ for some fixed $s\in \left( 0,1\right] ,$ and $r\geq 1$,
then%
\begin{eqnarray}
&& \\
&&\left\vert \frac{1}{2}\left[ \frac{f\left( u\right) +f\left( v\right) }{2}%
+f\left( \frac{u+v}{2}\right) \right] -\frac{1}{v-u}\int_{u}^{v}f\left(
x\right) dx\right\vert  \notag \\
&\leq &\frac{v-u}{8\left( 2^{s-1}\left( s+1\right) \left( s+2\right) \right)
^{1/r}}\left( \frac{1}{2}\right) ^{1-1/r}  \notag \\
&&\times \left\{ \left[ \left( \frac{3^{s+2}}{2^{s+1}}-\frac{%
2^{s+2}+s-2^{s+1}s+4}{2}\right) \left\vert f^{\prime }\left( u\right)
\right\vert ^{r}\allowbreak +\left( \frac{s}{2}+\frac{1}{2^{s+1}}\right)
\left\vert f^{\prime }\left( v\right) \right\vert ^{r}\right] ^{1/r}\right. 
\notag \\
&&+\left. \left[ \left( \frac{s}{2}+\frac{1}{2^{s+1}}\right) \left\vert
f^{\prime }\left( u\right) \right\vert ^{r}+\left( \frac{3^{s+2}}{2^{s+1}}-%
\frac{2^{s+2}+s-2^{s+1}s+4}{2}\right) \left\vert f^{\prime }\left( v\right)
\right\vert ^{r}\right] ^{1/r}\right\} ,  \notag
\end{eqnarray}%
\begin{eqnarray}
&& \\
&&\left\vert \frac{1}{3}\left[ f\left( u\right) +f\left( v\right) +f\left( 
\frac{u+v}{2}\right) \right] -\frac{1}{v-u}\int_{u}^{v}f\left( x\right)
dx\right\vert  \notag \\
&\leq &\frac{v-u}{8\left( 2^{s-1}\left( s+1\right) \left( s+2\right) \right)
^{1/r}}\left( \frac{1}{2}\right) ^{1-1/r}  \notag \\
&&\times \left\{ \left[ \left( \frac{2^{2s+5}}{3^{s+2}}-\frac{%
2^{s+2}+s-2^{s+2}s+5}{3}\right) \left\vert f^{\prime }\left( u\right)
\right\vert ^{r}\allowbreak +\left( \frac{s-1}{3}+\frac{2^{s+3}}{3^{s+2}}%
\right) \left\vert f^{\prime }\left( v\right) \right\vert ^{r}\right]
^{1/r}\right.  \notag \\
&&+\left. \left[ \left( \frac{s-1}{3}+\frac{2^{s+3}}{3^{s+2}}\right)
\left\vert f^{\prime }\left( u\right) \right\vert ^{r}+\left( \frac{2^{2s+5}%
}{3^{s+2}}-\frac{2^{s+2}+s-2^{s+2}s+5}{3}\right) \left\vert f^{\prime
}\left( v\right) \right\vert ^{r}\right] ^{1/r}\right\} ,  \notag
\end{eqnarray}%
\begin{eqnarray}
&& \\
&&\left\vert \frac{1}{6}\left[ f\left( u\right) +f\left( v\right) +4f\left( 
\frac{u+v}{2}\right) \right] -\frac{1}{v-u}\int_{u}^{v}f\left( x\right)
dx\right\vert  \notag \\
&\leq &\frac{v-u}{8\left( 2^{s-1}\left( s+1\right) \left( s+2\right) \right)
^{1/r}}\left( \frac{1}{2}\right) ^{1-1/r}  \notag \\
&&\times \left\{ \left[ \left( \frac{2\times 5^{s+2}}{3^{s+2}}-\frac{%
2^{s+3}+2s-2^{s+1}s+7}{3}\right) \left\vert f^{\prime }\left( u\right)
\right\vert ^{r}\allowbreak +\left( \frac{2s+1}{3}+\frac{2}{3^{s+2}}\right)
\left\vert f^{\prime }\left( v\right) \right\vert ^{r}\right] ^{1/r}\right. 
\notag \\
&&+\left. \left[ \left( \frac{2s+1}{3}+\frac{2}{3^{s+2}}\right) \left\vert
f^{\prime }\left( u\right) \right\vert ^{r}+\left( \frac{2\times 5^{s+2}}{%
3^{s+2}}-\frac{2^{s+3}+2s-2^{s+1}s+7}{3}\right) \left\vert f^{\prime }\left(
v\right) \right\vert ^{r}\right] ^{1/r}\right\} .  \notag
\end{eqnarray}
\end{corollary}

\begin{remark}
If setting $s=1$ in above Corollary, then we obtain the inequalities in \cite%
[pp. 6 and 7]{xi}.
\end{remark}

\begin{remark}
If setting $s=1$ and $r=1$ in above Corollary, then we obtain the
inequalities in \cite[(3.7)]{xi}. Hence, the results in above Corollary are
generalizations of the results of Xi and Qi in \cite{xi}.
\end{remark}

\section{\textbf{Applications}}

Let%
\begin{eqnarray*}
A\left( u,v\right) &=&\frac{u+v}{2},\text{ }L\left( u,v\right) =\frac{v-u}{%
\ln v-\ln u}\text{ \ \ \ \ }\left( u\neq v\right) , \\
L_{p}\left( u,v\right) &=&\left( \frac{v^{p+1}-u^{p+1}}{\left( p+1\right)
\left( v-u\right) }\right) ^{1/p},\text{ }u\neq v,\text{ }p\in 
\mathbb{R}
,\text{ }p\neq -1,0
\end{eqnarray*}%
be the arithmetic mean, logarithmic mean, generalized logarithmic mean for $%
u,v>0$ respectively.

\begin{criterion}
Let $g:%
\mathbb{R}
\rightarrow 
\mathbb{R}
_{+}$ be a non-negative convex function on $%
\mathbb{R}
$. Then $g^{s}(x)$ is $s$-convex on $I$, for some fixed $s\in \left(
0,1\right) $ (see \cite{alo}).
\end{criterion}

\begin{proposition}
\label{p1}Let $s\in \left( 0,1\right] ,$ $u,v>0,$ $1/p+1/r=1,$ $0\leq
\lambda ,$ $\mu \leq 1,$ then%
\begin{eqnarray}
&&\left\vert \frac{\lambda u^{s}+\mu v^{s}}{2}+\frac{2-\lambda -\mu }{2}%
A^{s}\left( u,v\right) -L_{s}^{s}\left( u,v\right) \right\vert  \notag \\
&\leq &\frac{v-u}{4}\left( \frac{\left( 1-\lambda \right) ^{p+1}+\lambda
^{p+1}}{p+1}\right) ^{\frac{1}{p}}\left( \frac{s\left( 2^{s+1}-1\right)
u^{r\left( s-1\right) }}{2^{s}\left( s+1\right) }+\frac{sv^{r\left(
s-1\right) }}{2^{s}\left( s+1\right) }\right) ^{\frac{1}{r}}  \notag \\
&&+\frac{v-u}{4}\left( \frac{\mu ^{p+1}+\left( 1-\mu \right) ^{p+1}}{p+1}%
\right) ^{\frac{1}{p}}\left( \frac{su^{r\left( s-1\right) }}{2^{s}\left(
s+1\right) }+\frac{s\left( 2^{s+1}-1\right) v^{r\left( s-1\right) }}{%
2^{s}\left( s+1\right) }\right) ^{\frac{1}{r}}  \notag
\end{eqnarray}%
In particular, when $\lambda =\mu =1,$ we have%
\begin{eqnarray}
&&\left\vert A\left( u^{s},v^{s}\right) -L_{s}^{s}\left( u,v\right)
\right\vert  \notag \\
&\leq &\frac{v-u}{4}\left( \frac{1}{p+1}\right) ^{\frac{1}{p}}\left( \frac{%
s\left( 2^{s+1}-1\right) u^{r\left( s-1\right) }}{2^{s}\left( s+1\right) }+%
\frac{sv^{r\left( s-1\right) }}{2^{s}\left( s+1\right) }\right) ^{\frac{1}{r}%
}  \notag \\
&&+\frac{v-u}{4}\left( \frac{1}{p+1}\right) ^{\frac{1}{p}}\left( \frac{%
su^{r\left( s-1\right) }}{2^{s}\left( s+1\right) }+\frac{s\left(
2^{s+1}-1\right) v^{r\left( s-1\right) }}{2^{s}\left( s+1\right) }\right) ^{%
\frac{1}{r}}  \notag
\end{eqnarray}
\end{proposition}

\begin{proof}
The claim follows from Theorem \ref{t0} applied to $s$-convex in the second
sense mapping $f\left( x\right) =x^{s}.$
\end{proof}

\begin{remark}
In Proposition \ref{p1}, if we take $\lambda =\mu =1/3,$ then Proposition %
\ref{p1} reduces to \cite[pp. 2198]{sari2}. Hence, the results in
Proposition \ref{p1} are generalizations of the results of Sarikaya et al.
in \cite{sari2}.
\end{remark}

\begin{proposition}
Let $s\in \left( 0,1\right] ,$ $u,v>0,$ $r\geq 1,$ $0\leq \lambda ,$ $\mu
\leq 1,$ then%
\begin{eqnarray*}
&&\left\vert \frac{\lambda u^{s}+\mu v^{s}}{2}+\frac{2-\lambda -\mu }{2}%
A^{s}\left( u,v\right) -L_{s}^{s}\left( u,v\right) \right\vert \\
&\leq &\frac{s\left( v-u\right) }{8}\left( \frac{1}{2^{s-1}\left( s+1\right)
\left( s+2\right) }\right) ^{1/r} \\
&&\times \left\{ \left( \left( 2\lambda ^{2}-2\lambda +1\right) \right)
^{1-1/r}\times \left[ u^{r\left( s-1\right) }E+v^{r\left( s-1\right) }L%
\right] ^{1/r}\right. \\
&&+\left. \left( \left( 2\mu ^{2}-2\mu +1\right) \right) ^{1-1/r}\times 
\left[ u^{r\left( s-1\right) }I+v^{r\left( s-1\right) }F\right]
^{1/r}\right\}
\end{eqnarray*}%
In particular, when $\lambda =\mu =1,$ we have%
\begin{eqnarray*}
&&\left\vert A\left( u^{s},v^{s}\right) -L_{s}^{s}\left( u,v\right)
\right\vert \\
&\leq &\frac{s\left( v-u\right) }{8}\left( \frac{1}{2^{s-1}\left( s+1\right)
\left( s+2\right) }\right) ^{1/r}\times \left\{ \left[ u^{r\left( s-1\right)
}\left( 1+s2^{s+1}\right) +v^{r\left( s-1\right) }\right] ^{1/r}\right. \\
&&+\left. \left[ u^{r\left( s-1\right) }+v^{r\left( s-1\right) }\left(
1+s2^{s+1}\right) \right] ^{1/r}\right\}
\end{eqnarray*}%
Moreover, when $r=1,$ we have%
\begin{equation*}
\left\vert A\left( u^{s},v^{s}\right) -L_{s}^{s}\left( u,v\right)
\right\vert \leq \frac{s^{2}\left( 1+s2^{s}\right) \left( v-u\right) }{%
2^{s}\left( s+1\right) \left( s+2\right) }\times A\left(
u^{s-1},v^{s-1}\right)
\end{equation*}
\end{proposition}

\begin{proof}
The claim follows from Theorem \ref{t1} applied to $s$-convex in the second
sense mapping $f\left( x\right) =x^{s}.$
\end{proof}

\begin{proposition}
\label{p3}\bigskip Let $s\in \left( 0,1\right] ,$ $u,v>0,$ $r\geq 1,$ $0\leq
\lambda ,$ $\mu \leq 1,$ then%
\begin{eqnarray}
&&  \label{31} \\
&&\left\vert \frac{\lambda a^{-s}+\mu v^{-s}}{2}+\frac{2-\lambda -\mu }{2}%
A^{-s}\left( u,v\right) -L_{-s}^{-s}\left( u,v\right) \right\vert  \notag \\
&\leq &\frac{s\left( v-u\right) }{8}\left( \frac{1}{2^{s-1}\left( s+1\right)
\left( s+2\right) }\right) ^{1/r}  \notag \\
&&\times \left\{ \left( \left( 2\lambda ^{2}-2\lambda +1\right) \right)
^{1-1/r}\times \left[ \left\vert u^{-r\left( s+1\right) }\right\vert
E+v^{-r\left( s+1\right) }L\right] ^{1/r}\right.  \notag \\
&&+\left. \left( \left( 2\mu ^{2}-2\mu +1\right) \right) ^{1-1/r}\times 
\left[ u^{-r\left( s+1\right) }I+v^{-r\left( s+1\right) }F\right]
^{1/r}\right\}  \notag
\end{eqnarray}%
In particular, when $\lambda =\mu =1/3,$ we have%
\begin{eqnarray*}
&&\left\vert \frac{1}{3}A\left( u^{-s},v^{-s}\right) +\frac{2}{3}%
A^{-s}\left( u,v\right) -L_{-s}^{-s}\left( u,v\right) \right\vert \leq \frac{%
s\left( v-u\right) }{8}\left( \frac{1}{2^{s-1}\left( s+1\right) \left(
s+2\right) }\right) ^{1/r}\left( \frac{5}{9}\right) ^{1-1/r} \\
&&\times \left\{ \left[ \left( \frac{2\times 5^{s+2}}{3^{s+2}}-\frac{%
2^{s+3}+2s-2^{s+1}s+7}{3}\right) \left\vert u^{-r\left( s+1\right)
}\right\vert +\left( \frac{2s+1}{3}+\frac{2}{3^{s+2}}\right) \left\vert
v^{-r\left( s+1\right) }\right\vert \right] ^{1/r}\right. \\
&&+\left. \left[ \left( \frac{2s+1}{3}+\frac{2}{3^{s+2}}\right) \left\vert
u^{-r\left( s+1\right) }\right\vert +\left( \frac{2\times 5^{s+2}}{3^{s+2}}-%
\frac{2^{s+3}+2s-2^{s+1}s+7}{3}\right) \left\vert v^{-r\left( s+1\right)
}\right\vert \right] ^{1/r}\right\}
\end{eqnarray*}%
Moreover, when $s=r=1,$ we obtain%
\begin{equation*}
\left\vert \frac{1}{3}A\left( u^{-1},v^{-1}\right) +\frac{2}{3}A^{-1}\left(
u,v\right) -L_{-1}^{-1}\left( u,v\right) \right\vert \leq \left( v-u\right) 
\frac{5}{36}A\left( u^{-2},v^{-2}\right) ,
\end{equation*}%
that it is the inequality in \cite[pp. 2199, line 2]{sari2}.
\end{proposition}

\begin{proof}
The claim follows from Theorem \ref{t1} applied to $s$-convex in the second
sense mapping $f\left( x\right) =\frac{1}{x^{s}}.$
\end{proof}

\begin{remark}
If we take $\lambda =\mu =1/3$ in Proposition \ref{p3}, then inequality (\ref%
{31}) reduces to \cite[pp. 2199]{sari2}. Hence, the results in Proposition %
\ref{p3} are generalizations of the results of Sarikaya et al. in \cite%
{sari2}.
\end{remark}

\end{document}